\documentclass[a4paper, 12pt,reqno]{amsart}

\usepackage[T2A]{fontenc}
\usepackage[cp1251]{inputenc}
\usepackage[ukrainian]{babel}
\usepackage{amsmath}
\usepackage{tikz}

\usepackage{amsmath}
 \usepackage{graphicx}
\usepackage{textcase}
\usepackage{amssymb}

\theoremstyle{plain}
\newtheorem{theorem}{Теорема}
\newtheorem*{theorem*}{Теорема}

\newtheorem*{corollary*}{Наслідок}
\newtheorem{lemma}{Лема}
\newtheorem*{lemma*}{Лема}

\newtheorem*{proposition*}{Твердження}

\newtheorem*{conjecture*}{Гіпотеза}
\theoremstyle{definition}
\newtheorem{definition}{Означення}
\newtheorem*{definition*}{Означення}
\theoremstyle{remark}
\newtheorem{remark}{Зауваження}
\newtheorem*{remark*}{Зауваження}

\begin{document}

УДК 517.5
\title[Узагальнені представлення дійсних чисел]{Про деякі узагальнення зображень дійсних чисел}
\author{ С. О. Сербенюк  }
\address{Інститут математики НАН України\\
         Київ, Україна}
\email{simon6@ukr.net}

\maketitle

\begin{abstract}

The article is devoted to introduce and investigation of new  numeral systems. These representations are generalization of classical representations of real numbers. The main properties of investigating representation are described.

The results are represented in International Conference of Yung Mathematicians, June 3-6, 2015, Kyiv.

\end{abstract}

\section{Знакозмінні представлення дійсних чисел}
\subsection{Квазі-нега-s-ве представлення чисел.}
Дослідимо можливість представлення дійсних чисел знакозмінним  s-им розкладом, знакозмінним рядом Кантора, квазі-нега-$\tilde Q$-розкладом.

Для початку розглянемо найпростіший випадок  --- узагальнення s-их розкладів

\begin{equation}
\label{eq: znakozminnyi-s-rozklad 1}
\sum^{\infty} _{n=1} {\frac{(-1)^{\delta_n}\alpha_n}{s^n}},
\end{equation}
де $s$ ---  фіксоване натуральне число, більше $1$, $\alpha_n\in A\equiv \{0,1,...,s-1\}$.

 Очевидно, залежність $\delta_n$  може бути  або  функцією (що набуває лише значень з множини цілих додатних чисел)  натурального аргументу $\delta_n=\delta(n)$, або залежністю $\delta_n=\delta(\alpha_n)$ від $\alpha_n$. Розглянемо перший випадок, оскільки у другому випадку  легко змоделювати  приклади, де розклад \eqref{eq: znakozminnyi-s-rozklad 1} не є представленням дійсних чисел з деякого відрізка.
 
 Нехай  маємо фіксовану підмножину $N_B$ множини натуральних чисел, $B=(b_n)$ --- зростаюча  послідовність елементів множини $N_B$. Означимо
$$
\delta_n=\begin{cases}
2,&\text{якщо $n \notin N_B$;}\\
1,&\text{якщо $n \in N_B$.}
\end{cases}
$$

Збіжність ряду \eqref{eq: znakozminnyi-s-rozklad 1} є очевидною. Введемо деякі допоміжні поняття, необхідні при подальшому вивченні рядів \eqref{eq: znakozminnyi-s-rozklad 1}.

Якщо деяке число $x$ можна представити  у вигляді  розкладу \eqref{eq: znakozminnyi-s-rozklad 1}, то  цей факт формально  позначатимемо $x=\Delta^{(\pm s, B)} _{\alpha_1\alpha_2...\alpha_n...}$, а останній запис називатимемо \emph{зображенням  числа $x$ знакозмінним $s$-$N_B$-розкладом} або \emph{квазі-нега-s-им зображенням числа $x$}. 

  Відповідно, ряд  \eqref{eq: znakozminnyi-s-rozklad 1} називатимемо \emph{знакозмінним $s$-$N_B$-розкладом числа $x$ } або \emph{квазі-нега-s-им представленням числа $x$}. Тобто,
\begin{equation}
\label{def: znakozminnyi-s-rozklad 1}
x=\Delta^{(\pm s, B)} _{\alpha_1\alpha_2...\alpha_n...}\equiv \sum^{\infty} _{n=1} {\frac{(-1)^{\delta_n}\alpha_n}{s^n}},
\end{equation}
де $s$ ---  фіксоване натуральне число, більше $1$, $\alpha_n\in A\equiv \{0,1,...,s-1\}$,
$$
\delta_n=\begin{cases}
2,&\text{якщо $n \notin N_B$;}\\
1,&\text{якщо $n \in N_B$.}
\end{cases}
$$

Очевидним є факт того, що
$$
x^{'} _0=\inf{\sum^{\infty} _{n=1} {\frac{(-1)^{\delta_n}\alpha_n}{s^n}}}=-\sum^{\infty}_{n=1}{\frac{s-1}{s^{b_n}}}\equiv \Delta^{(\pm s, B)} _{\beta_1\beta_2...\beta_n...},
$$
де
\begin{equation}
\label{eq: znakozminnyi-s-rozklad 2}
\beta_n=\begin{cases}
0,&\text{якщо $n \notin N_B$;}\\
s-1,&\text{якщо $n \in N_B$.}
\end{cases}
\end{equation}
$$
x^{''} _0=\sup{\sum^{\infty} _{n=1} {\frac{(-1)^{\delta_n}\alpha_n}{s^n}}}=\sum_{n \notin N_B} {\frac{s-1}{s^n}}=1-\sum^{\infty}_{n=1}{\frac{s-1}{s^{b_n}}}\equiv \Delta^{(\pm s, B)} _{\gamma_1\gamma_2...\gamma_n...},
$$
де
\begin{equation}
\label{eq: znakozminnyi-s-rozklad 3}
\gamma_n=\begin{cases}
s-1,&\text{якщо $n \notin N_B$;}\\
0,&\text{якщо $n \in N_B$.}
\end{cases}
\end{equation}

\begin{theorem}
\label{th: znakozminnyi-s-rozklad 1}
Для  будь-якого числа $x\in[x^{'}_0;x^{''} _0]$  існує послідовність  $(\alpha_n)$,   $\alpha_n \in A$, така, що    $x$   подається   квазі-нега-s-им представленням \eqref{def: znakozminnyi-s-rozklad 1}.
\end{theorem}
\begin{proof}
Нехай $(x^{'}_0;x^{''} _0)\ni x$ --- довільне число. Якщо  $1\notin N_B$, тоді
$$
\frac{\alpha_1}{s}-\sum^{\infty} _{n=1} {\frac{s-1}{s^{b_n}}} \le x < \frac{\alpha_1}{s}+\sum_{1<n \notin N_B} {\frac{s-1}{s^n}},
$$
$$
-\sum^{\infty} _{n=1} {\frac{s-1}{s^{b_n}}} \le x -\frac{\alpha_1}{s}<\sum_{1<n \notin N_B} {\frac{s-1}{s^n}}=1-\frac{s-1}{s}-\sum^{\infty} _{n=1} {\frac{s-1}{s^{b_n}}}=\frac{1}{s}-\sum^{\infty} _{n=1} {\frac{s-1}{s^{b_n}}},
$$
та, якщо $1\in N_B$,
$$
-\frac{\alpha_1}{s}+\frac{s-1}{s}-\sum^{\infty} _{n=1}{\frac{s-1}{s^{b_n}}}=-\frac{\alpha_1}{s}-\sum^{\infty} _{n=2}{\frac{s-1}{s^{b_n}}}<x\le -\frac{\alpha_1}{s}+1-\sum^{\infty} _{n=1}{\frac{s-1}{s^{b_n}}},
$$
$$
\frac{s-1}{s}-\sum^{\infty} _{n=1}{\frac{s-1}{s^{b_n}}}<x+\frac{\alpha_1}{s} \le  1-\sum^{\infty} _{n=1}{\frac{s-1}{s^{b_n}}}.
$$
В силу того, що
$$
[x^{'}_0;x^{''} _0]=\bigcup^{s-1}_{i=0} {\left[\frac{i}{s}-\sum^{\infty} _{n=1} {\frac{s-1}{s^{b_n}}};\frac{i}{s}+\frac{1}{s}-\sum^{\infty} _{n=1} {\frac{s-1}{s^{b_n}}}\right]},
$$
якщо $1 \notin N_B$, та
$$
[x^{'}_0;x^{''} _0]=\bigcup^{s-1}_{i=0} {\left[-\frac{i}{s}+\frac{s-1}{s}-\sum^{\infty} _{n=1} {\frac{s-1}{s^{b_n}}};-\frac{i}{s}+1-\sum^{\infty} _{n=1} {\frac{s-1}{s^{b_n}}}\right]},
$$
для $1 \in N_B$, тому, якщо для випадку $1 \notin N_B$ позначити $x-\frac{\alpha_1}{s}=x_1$, отримаємо випадки:
\begin{enumerate}
\item  
$$
x_1=-\sum^{\infty} _{n=1} {\frac{s-1}{s^{b_n}}}, ~\text{що еквівалентно}~x=\Delta^{(\pm s, B)} _{\alpha_1\beta_2\beta_3...\beta_n...},
$$
де елементи послідовності $(\beta_n)$ задовольняють умову \eqref{eq: znakozminnyi-s-rozklad 2};
\item
$$
x_1\ne -\sum^{\infty} _{n=1} {\frac{s-1}{s^{b_n}}}. ~\text{В такому разі}~x=\frac{\alpha_1}{s}+x_1.
$$
Оцінимо значення $x_1$ трохи пізніше.
\end{enumerate}

Якщо для випадку $1 \in N_B$ позначити $x+\frac{\alpha_1}{s}=x_1$, отримаємо випадки:
\begin{enumerate}
\item  
$$
x_1=1-\sum^{\infty} _{n=1} {\frac{s-1}{s^{b_n}}}, ~\text{що еквівалентно}~x=\Delta^{(\pm s, B)} _{\alpha_1\gamma_2\gamma_3...\gamma_n...},
$$
де елементи послідовності $(\gamma_n)$ задовольняють умову \eqref{eq: znakozminnyi-s-rozklad 3};
\item
$$
x_1\ne 1 -\sum^{\infty} _{n=1} {\frac{s-1}{s^{b_n}}}. ~\text{В такому разі}~x=-\frac{\alpha_1}{s}+x_1.
$$
\end{enumerate}

Нехай $2 \notin N_B$. Тоді
$$
\frac{\alpha_2}{s^2}-\sum_{b_n>2} {\frac{s-1}{s^{b_n}}}\le x_1<\frac{\alpha_2}{s^2}+\sum_{2<n \notin N_B}{\frac{s-1}{s^n}}.
$$
При $2 \in N_B$
$$
-\frac{\alpha_2}{s^2}-\sum_{b_n>2} {\frac{s-1}{s^{b_n}}}< x_1\le -\frac{\alpha_2}{s^2}+\sum_{2<n \notin N_B}{\frac{s-1}{s^n}}.
$$

Позначивши $x_2=x_1-\frac{\alpha_2}{s^2}$ для випадку $2 \notin N_B$ та $x_2=x_1+\frac{\alpha_2}{s^2}$ для $2 \in N_B$, отримаємо: 
\begin{itemize}
\item для  $2 \notin N_B$
$$
-\sum_{b_n>2} {\frac{s-1}{s^{b_n}}}\le x_2<\sum_{2<n \notin N_B}{\frac{s-1}{s^n}},
$$
звідки слідує два випадки:
\begin{enumerate}
\item якщо
$$
x_2=-\sum^{\infty} _{b_n>2} {\frac{s-1}{s^{b_n}}}, ~\text{то}~x=\Delta^{(\pm s, B)} _{\alpha_1\alpha_2\beta_3...\beta_n...},
$$
де елементи послідовності $(\beta_n)$ задовольняють умову \eqref{eq: znakozminnyi-s-rozklad 2};
\item якщо
$$
x_2\ne -\sum^{\infty} _{b_n>2} {\frac{s-1}{s^{b_n}}}, ~\text{то}~x=\frac{(-1)^{\delta_1}\alpha_1}{s}+\frac{(-1)^{\delta_2}\alpha_1}{s^2}+x_2;
$$
\end{enumerate}

\item для $2 \in N_B$
$$
-\sum_{b_n>2} {\frac{s-1}{s^{b_n}}}< x_2\le \sum_{2<n \notin N_B}{\frac{s-1}{s^n}},
$$
звідки слідує, що:
\begin{enumerate}
\item  якщо
$$
x_2=\sum_{2<n\notin N_B} {\frac{s-1}{s^n}}, ~\text{то}~x=\Delta^{(\pm s, B)} _{\alpha_1\alpha_2\gamma_3\gamma_4...\gamma_n...},
$$
де елементи послідовності $(\gamma_n)$ задовольняють умову \eqref{eq: znakozminnyi-s-rozklad 3};
\item якщо
$$
x_2\ne\sum_{2<n\notin N_B} {\frac{s-1}{s^n}}, ~\text{то}~x=\frac{(-1)^{\delta_1}\alpha_1}{s}+\frac{(-1)^{\delta_2}\alpha_1}{s^2}+x_2.
$$
\end{enumerate}
\end{itemize}

За скінченну кількість кроків $m$ отримаємо, що
$$
\frac{(-1)^{\delta_{m+1}}\alpha_{m+1}}{s^{m+1}}-\sum_{b_n>m+1}{\frac{s-1}{s^{b_n}}}<x_m<\frac{(-1)^{\delta_{m+1}}\alpha_{m+1}}{s^{m+1}}+\sum_{m+1<n \notin N_B} {\frac{s-1}{s^n}},
$$
$$
-\sum_{b_n>m+1}{\frac{s-1}{s^{b_n}}}<x_{m+1}<\sum_{m+1<n \notin N_B} {\frac{s-1}{s^n}},
$$
аналогічно
$$
x_{m+1}=\begin{cases}
-\sum_{b_n>m+1}{\frac{s-1}{s^{b_n}}},&\text{якщо $m+1 \notin N_B$;}\\
\sum_{m+1<n\notin N_B}{\frac{s-1}{s^n}},&\text{якщо $m+1 \in N_B$.}
\end{cases}
$$

Якщо одна з умов останньої  системи справджується, то 
$$
x=\Delta^{(\pm s, B)} _{\alpha_1\alpha_2...\alpha_{m+1}\beta_{m+2}\beta_{m+3}...}~\text{або}~x=\Delta^{(\pm s, B)} _{\alpha_1\alpha_2...\alpha_{m+1}\gamma_{m+2}\gamma_{m+3}...}.
$$

В іншому випадку, продовжуючи процес до нескінченності, отримаємо
$$
x=\frac{(-1)^{\delta_1}\alpha_1}{s}+x_1=\frac{(-1)^{\delta_1}\alpha_1}{s}+\frac{(-1)^{\delta_2}\alpha_2}{s^2}+x_2=...=
$$
$$
=\frac{(-1)^{\delta_1}\alpha_1}{s}+\frac{(-1)^{\delta_2}\alpha_2}{s^2}+\frac{(-1)^{\delta_3}\alpha_3}{s^3}+...+\frac{(-1)^{\delta_n}\alpha_n}{s^n}+x_n=....
$$

Звідки слідує, що
$$
x=\sum^{\infty} _{n=1}{\frac{(-1)^{\delta_n}\alpha_n}{s^n}}.
$$
\end{proof}

\begin{lemma} Справедливою є наступна тотожність
$$
\Delta^{(\pm s, B)} _{\alpha_1\alpha_2...\alpha_n...}+\sum^{\infty} _{n=1} {\frac{s-1}{s^{b_n}}}\equiv \Delta^s _{\alpha^{'} _1\alpha^{'} _2...\alpha^{'} _n...},
$$
де
$$
\alpha^{'} _n=\begin{cases}
\alpha_n,&\text{якщо $n \notin N_B$;}\\
s-1-\alpha_n,&\text{якщо $n \in N_B$,}
\end{cases}
$$
$\Delta^s _{\alpha^{'} _1\alpha^{'} _2...\alpha^{'} _n...}$ --- s-кове зображення дійсного числа з $[0;1]$.
\end{lemma}

Наслідком останньої леми та теореми \ref{th: znakozminnyi-s-rozklad 1} є наступне твердження.

\begin{theorem}
Числа з відрізка $[x^{'} _0;x^{''} _0]$ мають не більше двох квазі-нега-s-кових зображень, а саме:
\begin{enumerate}
\item якщо $n,n+1 \in N_B$, то
$$
\Delta^{(\pm s, B)} _{\alpha_1\alpha_2...\alpha_{n-1}[s-1-\alpha_n][s-1]\beta{n+2}\beta_{n+3}...}=\Delta^{(\pm s, B)} _{\alpha_1\alpha_2...\alpha_{n-1}[s-\alpha_n]0\gamma_{n+2}\gamma_{n+3}...};
$$
\item якщо $n \in N_B$, $n+1 \notin N_B$, то
$$
\Delta^{(\pm s, B)} _{\alpha_1\alpha_2...\alpha_{n-1}[s-1-\alpha_n]0\beta{n+2}\beta_{n+3}...}=\Delta^{(\pm s, B)} _{\alpha_1\alpha_2...\alpha_{n-1}[s-\alpha_n][s-1]\gamma_{n+2}\gamma_{n+3}...};
$$
\item якщо $n \notin N_B$, $n+1 \in N_B$, то
$$
\Delta^{(\pm s, B)} _{\alpha_1\alpha_2...\alpha_{n-1}\alpha_n[s-1]\beta{n+2}\beta_{n+3}...}=\Delta^{(\pm s, B)} _{\alpha_1\alpha_2...\alpha_{n-1}[\alpha_n-1]0\gamma_{n+2}\gamma_{n+3}...};
$$
\item якщо $n \notin N_B$, $n+1 \notin N_B$, то
$$
\Delta^{(\pm s, B)} _{\alpha_1\alpha_2...\alpha_{n-1}\alpha_n0\beta{n+2}\beta_{n+3}...}=\Delta^{(\pm s, B)} _{\alpha_1\alpha_2...\alpha_{n-1}[\alpha_n-1][s-1]\gamma_{n+2}\gamma_{n+3}...},
$$
\end{enumerate}
де $\alpha_n\ne 0$, $(\beta_n)$ та $(\gamma_n)$ --- послідовності цифр $0$ та $s-1$, які задовольняють умови \eqref{eq: znakozminnyi-s-rozklad 2} і \eqref{eq: znakozminnyi-s-rozklad 3} відповідно, мають числа зі зліченної підмножини $[x^{'} _0;x^{''} _0]$.
\end{theorem}
\begin{proof} Справді,
\begin{enumerate}
\item нехай $n,n+1 \in N_B$, тоді
$$
\Delta^{(\pm s, B)} _{\alpha_1\alpha_2...\alpha_{n-1}[s-1-\alpha_n][s-1]\beta{n+2}\beta_{n+3}...}-\Delta^{(\pm s, B)} _{\alpha_1\alpha_2...\alpha_{n-1}[s-\alpha_n]0\gamma_{n+2}\gamma_{n+3}...}=
$$
$$
\left(-\frac{s-1-\alpha_n}{s^n}-\frac{s-1}{s^{n+1}}-\sum_{b_n>n+1}{\frac{s-1}{s^{b_n}}}\right)-\left(-\frac{s-\alpha_n}{s^n}+\sum_{n+1<k\notin N_B}{\frac{s-1}{s^k}}\right)=
$$
$$
=\frac{1}{s^n}-\frac{s-1}{s^{n+1}}-\sum^{\infty} _{k=n+2}{\frac{s-1}{s^k}}=0;
$$
\item нехай $n \in N_B, n+1 \notin N_B$, тоді
$$
\Delta^{(\pm s, B)} _{\alpha_1\alpha_2...\alpha_{n-1}[s-1-\alpha_n]0\beta{n+2}\beta_{n+3}...}-\Delta^{(\pm s, B)} _{\alpha_1\alpha_2...\alpha_{n-1}[s-\alpha_n][s-1]\gamma_{n+2}\gamma_{n+3}...}=
$$
$$
=\left(-\frac{s-1-\alpha_n}{s^n}-\sum_{b_n>n+1}{\frac{s-1}{s^{b_n}}}\right)-\left(-\frac{s-\alpha_n}{s^n}+\frac{s-1}{s^{n+1}}+\sum_{n+1<k\notin N_B}{\frac{s-1}{s^k}}\right)=
$$
$$
=\frac{s-\alpha_n}{s^n}+\frac{1}{s^n}+\frac{s-\alpha_n}{s^n}-\frac{s-1}{s^{n+1}}-\frac{1}{s^{n+1}}=0;
$$
\item нехай $n\notin N_B, n+1 \in N_B$, тоді
$$
\Delta^{(\pm s, B)} _{\alpha_1\alpha_2...\alpha_{n-1}\alpha_n[s-1]\beta{n+2}\beta_{n+3}...}-\Delta^{(\pm s, B)} _{\alpha_1\alpha_2...\alpha_{n-1}[\alpha_n-1]0\gamma_{n+2}\gamma_{n+3}...}=
$$
$$
=\left(\frac{\alpha_n}{s^n}-\frac{s-1}{s^{n+1}}-\sum_{b_n>n+1}{\frac{s-1}{s^{b_n}}}\right)-\left(\frac{\alpha_n-1}{s^n}+\sum_{n+1<k\notin N_B}{\frac{s-1}{s^k}}\right)=
$$
$$
\frac{\alpha_n}{s^n}-\frac{s-1}{s^{n+1}}-\frac{\alpha_n-1}{s^n}-\frac{1}{s^{n+1}}=0;
$$
\item нехай $n\notin N_B,n+1 \notin N_B$, тоді
$$
\Delta^{(\pm s, B)} _{\alpha_1\alpha_2...\alpha_{n-1}\alpha_n0\beta{n+2}\beta_{n+3}...}-\Delta^{(\pm s, B)} _{\alpha_1\alpha_2...\alpha_{n-1}[\alpha_n-1][s-1]\gamma_{n+2}\gamma_{n+3}...}=
$$
$$
=\left(\frac{\alpha_n}{s^n}-\sum_{b_n>n+1}{\frac{s-1}{s^{b_n}}}\right)-\left(\frac{\alpha_n-1}{s^n}+\frac{s-1}{s^{n+1}}+\sum_{n+1<k\notin N_B}{\frac{s-1}{s^k}}\right)=
$$
$$
=\frac{\alpha_n}{s^n}-\frac{\alpha_n-1}{s^n}-\frac{s-1}{s^{n+1}}-\frac{1}{s^{n+1}}=0.
$$
\end{enumerate}
\end{proof}

\emph{Циліндром $\Delta^{(\pm s, B)} _{c_1c_2...c_n}$ рангу $n$  з основою $c_1c_2...c_n$} називається множина виду
$$
\Delta^{(\pm s, B)} _{c_1c_2...c_n}\equiv\left\{x: x=\sum^{n} _{i=1}{\frac{(-1)^{\delta_i}c_i}{s^i}}+\sum^{\infty} _{j=n+1}{\frac{(-1)^{\delta_j}\alpha_j}{s^j}}\right\},
$$
де $c_1,c_2,...,c_n$ --- фіксовані числа з множини $A\equiv \{0,1,...,s-1\}$, $\alpha_j \in A$.

\begin{lemma} Для циліндрів $\Delta^{(\pm s, B)} _{c_1c_2...c_n}$ справедливими є наступні властивості:
\begin{enumerate}
\item Циліндр $\Delta^{(\pm s, B)} _{c_1c_2...c_n}$ є відрізком,  причому
$$
\Delta^{(\pm s, B)} _{c_1c_2...c_n}\equiv \left[\sum^{n} _{i=1}{\frac{(-1)^{\delta_i}c_i}{s^i}}-\sum_{n<b_k}{\frac{s-1}{s^{b_k}}};\sum^{n} _{i=1}{\frac{(-1)^{\delta_i}c_i}{s^i}}+\sum_{n<k\notin N_B}{\frac{s-1}{s^k}}\right].
$$
\item
$$
|\Delta^{(\pm s, B)} _{c_1c_2...c_n}|=\frac{1}{s^n}.
$$
\item 
$$
\Delta^{(\pm s, B)} _{c_1c_2...c_nc}\subset \Delta^{(\pm s, B)} _{c_1c_2...c_n}.
$$
\item 
$$
\Delta^{(\pm s, B)} _{c_1c_2...c_n}=\bigcup^{s-1} _{c=0}{\Delta^{(\pm s, B)} _{c_1c_2...c_nc}}.
$$
\item
$$
\left\{
\begin{aligned}
\sup{\Delta^{(\pm s, B)} _{c_1c_2...c_{n-1}c_n}}&=\inf{\Delta^{(\pm s, B)} _{c_1c_2...c_{n-1}[c_n+1]}},&\text{якщо $n \notin N_B$;}\\
\inf{\Delta^{(\pm s, B)} _{c_1c_2...c_{n-1}c_n}}& = \sup{\Delta^{(\pm s, B)} _{c_1c_2...c_{n-1}[c_n+1]}},&\text{якщо $n \in N_B$},\\
\end{aligned}
\right.
$$
де $c_n \ne s-1$.
\item
$$
\bigcap^{\infty} _{n=1}{\Delta^{(\pm s, B)} _{c_1c_2...c_n}}=x=\Delta^{(\pm s, B)} _{c_1c_2...c_n...}.
$$
\end{enumerate}
\end{lemma}
\begin{proof}
Доведемо, що циліндр $\Delta^{(\pm s, B)} _{c_1c_2...c_n}$ є відрізком. Нехай $x \in \Delta^{(\pm s, B)} _{c_1c_2...c_n}$, тоді  
$$
 x^{'}=\sum^{n} _{i=1}{\frac{(-1)^{\delta_i}c_i}{s^i}}-\sum_{n<b_k}{\frac{s-1}{s^{b_k}}}\le x\le\sum^{n} _{i=1}{\frac{(-1)^{\delta_i}c_i}{s^i}}+\sum_{n<k\notin N_B}{\frac{s-1}{s^k}}=x^{''}.
$$

Отже, $\Delta^{(\pm s, B)} _{c_1c_2...c_n}\subseteq [x^{'};x^{''}]\ni x$.

В силу того, що 
$$
-\sum_{n<b_k}{\frac{s-1}{s^{b_k}}}=\inf{\sum^{\infty} _{j=n+1}{\frac{(-1)^{\delta_j}\alpha_j}{s^j}}},
$$
$$
\sum_{n<k\notin N_B}{\frac{s-1}{s^k}}=\sup{\sum^{\infty} _{j=n+1}{\frac{(-1)^{\delta_j}\alpha_j}{s^j}}},
$$
слідує, що $x \in\Delta^{(\pm s, B)} _{c_1c_2...c_n}$, $x^{'}\in\Delta^{(\pm s, B)} _{c_1c_2...c_n}\ni x^{''}$. Отже, $\Delta^{(\pm s, B)} _{c_1c_2...c_n}$ --- відрізок.

\emph{Властивості 2 --- 4, 6} є очевидними і слідують з наведених вище міркувань. 

5. Розглянемо різниці
$$
\sup{\Delta^{(\pm s, B)} _{c_1c_2...c_{n-1}c_n}}-\inf{\Delta^{(\pm s, B)} _{c_1c_2...c_{n-1}[c_n+1]}}=
$$
$$
=\frac{(-1)^{\delta_n}c_n}{s^n}+\sum_{n<k\notin N_B}{\frac{s-1}{s^k}}+\sum_{n<b_k}{\frac{s-1}{s^{b_k}}}-\frac{(-1)^{\delta_n}(c_n+1)}{s^n}=\frac{1}{s^n}+(-1)^{\delta_n}\frac{c_n-c_n-1}{s^n}=0,
$$
оскільки $n\notin N_B$.
$$
\inf{\Delta^{(\pm s, B)} _{c_1c_2...c_{n-1}c_n}}-\sup{\Delta^{(\pm s, B)} _{c_1c_2...c_{n-1}[c_n+1]}}=\frac{(-1)^{\delta_n}c_n}{s^n}-\sum_{n<b_k}{\frac{s-1}{s^{b_k}}}-\frac{(-1)^{\delta_n}(c_n+1)}{s^n}-\sum_{n<k\notin N_B}{\frac{s-1}{s^k}}=
$$
$$
=(-1)^{\delta_n}\frac{c_n-c_n-1}{s^n}-\frac{1}{s^n}=0,~\text{оскільки}~n\in N_B.
$$
\end{proof}

\subsection{Квазі-нега-D-представлення дійсних чисел.}
Поняття квазінега-s-го представлення можна узагальнити до квазі-нега-D-представлення (знакозмінного ряду Кантора)  
\begin{equation}
\label{eq: znakozminnyi ryad Kantora}
\sum^{\infty} _{n=1}{\frac{(-1)^{\delta_n}\varepsilon_n}{d_1d_2...d_n}},~\text{де}
\end{equation}
 $D\equiv (d_n)$ --- фіксована послідовність натуральних чисел, більших $1$, $\varepsilon_n\in A_{d_n}\equiv$ $\equiv \{0,1,...,d_n-~1\}$,
$$
\delta_n=\begin{cases}
2,&\text{якщо $n \notin N_B$;}\\
1,&\text{якщо $n \in N_B$,}
\end{cases}
$$

Позначення $x=\Delta^{(\pm D, B)} _{\varepsilon_1\varepsilon_2...\varepsilon_n...}$ називатимемо \emph{зображенням  числа $x$ знакозмінним $N_B$-розкладом в ряд Кантора} або \emph{квазі-нега-D-зображенням числа $x$}, а ряд \eqref{eq: znakozminnyi ryad Kantora} називатимемо \emph{знакозмінним $N_B$-розкладом в ряд Кантора} або \emph{квазі-нега-D-представленням числа $x$}.

Провівши міркування, подібні міркуванням в попередньому пункті, отримаємо:
\begin{theorem}
Для  будь-якого числа $x\in[x^{'}_0;x^{''} _0]$, де
$$
x^{'}_0=-\sum_{n\in N_B}{\frac{d_n-1}{d_1d_2...d_n}},~~~x^{''} _0=\sum_{n\notin N_B}{\frac{d_n-1}{d_1d_2...d_n}}
$$
 існує послідовність  $(\varepsilon_n)$,   $\varepsilon_n \in A_{d_n}$, така, що    
$$
x=\sum^{\infty} _{n=1}{\frac{(-1)^{\delta_n}\varepsilon_n}{d_1d_2...d_n}}\equiv \Delta^{(\pm D, B)} _{\varepsilon_1\varepsilon_2...\varepsilon_n...}
$$
\end{theorem}

\begin{lemma} Справедливою є наступна тотожність
$$
\Delta^{(\pm D, B)} _{\varepsilon_1\varepsilon_2...\varepsilon_n...}+\sum_{n\in N_B} {\frac{d_n-1}{d_1d_2...d_n}}\equiv \Delta^D _{\varepsilon^{'} _1\varepsilon^{'} _2...\varepsilon^{'} _n...},
$$
де
$$
\varepsilon^{'} _n=\begin{cases}
\varepsilon_n,&\text{якщо $n \notin N_B$;}\\
d_n-1-\varepsilon_n,&\text{якщо $n \in N_B$,}
\end{cases}
$$
$\Delta^D _{\varepsilon^{'} _1\varepsilon^{'} _2...\varepsilon^{'} _n...}$ --- D-зображення дійсного числа з $[0;1]$.
\end{lemma}

\begin{theorem}
Числа з відрізка $[x^{'} _0;x^{''} _0]$ мають не більше двох квазі-нега-D-зображень, а саме:
\begin{enumerate}
\item якщо $n,n+1 \in N_B$, то
$$
\Delta^{(\pm D, B)} _{\varepsilon_1\varepsilon_2...\varepsilon_{n-1}[d_n-1-\varepsilon_n][d_{n+1}-1]\beta{n+2}\beta_{n+3}...}=\Delta^{(\pm D, B)} _{\varepsilon_1\varepsilon_2...\varepsilon_{n-1}[d_n-\varepsilon_n]0\gamma_{n+2}\gamma_{n+3}...};
$$
\item якщо $n \in N_B$, $n+1 \notin N_B$, то
$$
\Delta^{(\pm D, B)} _{\varepsilon_1\varepsilon_2...\varepsilon_{n-1}[d_n-1-\varepsilon_n]0\beta{n+2}\beta_{n+3}...}=\Delta^{(\pm D, B)} _{\varepsilon_1\varepsilon_2...\varepsilon_{n-1}[d_n-\varepsilon_n][d_{n+1}-1]\gamma_{n+2}\gamma_{n+3}...};
$$
\item якщо $n \notin N_B$, $n+1 \in N_B$, то
$$
\Delta^{(\pm D, B)} _{\varepsilon_1\varepsilon_2...\varepsilon_{n-1}\varepsilon_n[d_{n+1}-1]\beta{n+2}\beta_{n+3}...}=\Delta^{(\pm D, B)} _{\varepsilon_1\varepsilon_2...\varepsilon_{n-1}[\varepsilon_n-1]0\gamma_{n+2}\gamma_{n+3}...};
$$
\item якщо $n \notin N_B$, $n+1 \notin N_B$, то
$$
\Delta^{(\pm D, B)} _{\varepsilon_1\varepsilon_2...\varepsilon_{n-1}\varepsilon_n0\beta{n+2}\beta_{n+3}...}=\Delta^{(\pm D, B)} _{\varepsilon_1\varepsilon_2...\varepsilon_{n-1}[\varepsilon_n-1][d_{n+1}-1]\gamma_{n+2}\gamma_{n+3}...},
$$
\end{enumerate}
де $\varepsilon_n\ne 0$, $(\beta_n)$ та $(\gamma_n)$ --- послідовності цифр $0$ та $d_n-1$, які задовольняють умови \eqref{eq: znakozminnyi-D-rozklad 2} і \eqref{eq: znakozminnyi-D-rozklad 3} відповідно, мають числа зі зліченної підмножини $[x^{'} _0;x^{''} _0]$.
\end{theorem}
\begin{equation}
\label{eq: znakozminnyi-D-rozklad 2}
\beta_n=\begin{cases}
0,&\text{якщо $n \notin N_B$;}\\
d_n-1,&\text{якщо $n \in N_B$,}
\end{cases}
\end{equation}
\begin{equation}
\label{eq: znakozminnyi-D-rozklad 3}
\gamma_n=\begin{cases}
d_n-1,&\text{якщо $n \notin N_B$;}\\
0,&\text{якщо $n \in N_B$.}
\end{cases}
\end{equation}

\emph{Циліндром $\Delta^{(\pm D, B)} _{c_1c_2...c_n}$ рангу $n$  з основою $c_1c_2...c_n$} називається множина виду
$$
\Delta^{(\pm D, B)} _{c_1c_2...c_n}\equiv\left\{x: x=\sum^{n} _{i=1}{\frac{(-1)^{\delta_i}c_i}{d_1d_2...d_i}}+\sum^{\infty} _{j=n+1}{\frac{(-1)^{\delta_j}\varepsilon_j}{d_1d_2...d_j}}\right\},
$$
де $c_i\in A_{d_i}$ --- фіксовані числа, $i=\overline{1,n}$,  та $\varepsilon_j \in A_{d_j}$.

\begin{lemma} Для циліндрів $\Delta^{(\pm D, B)} _{c_1c_2...c_n}$ справедливими є наступні властивості:
\begin{enumerate}
\item Циліндр $\Delta^{(\pm D, B)} _{c_1c_2...c_n}$ є відрізком,  причому
$$
\Delta^{(\pm s, B)} _{c_1c_2...c_n}\equiv \left[\sum^{n} _{i=1}{\frac{(-1)^{\delta_i}c_i}{d_1d_2...d_i}}-\sum_{n<k\in N_B}{\frac{d_k-1}{d_1d_2...d_k}};\sum^{n} _{i=1}{\frac{(-1)^{\delta_i}c_i}{d_1d_2...d_i}}+\sum_{n<k\notin N_B}{\frac{d_k-1}{d_1d_2...d_k}}\right].
$$
\item
$$
|\Delta^{(\pm D, B)} _{c_1c_2...c_n}|=\frac{1}{d_1d_2...d_n}.
$$
\item 
$$
\Delta^{(\pm D, B)} _{c_1c_2...c_nc}\subset \Delta^{(\pm D, B)} _{c_1c_2...c_n}.
$$
\item 
$$
\Delta^{(\pm D, B)} _{c_1c_2...c_n}=\bigcup^{d_{n+1}-1} _{c=0}{\Delta^{(\pm D, B)} _{c_1c_2...c_nc}}.
$$
\item
$$
\left\{
\begin{aligned}
\sup{\Delta^{(\pm D, B)} _{c_1c_2...c_{n-1}c_n}}&=\inf{\Delta^{(\pm D, B)} _{c_1c_2...c_{n-1}[c_n+1]}},&\text{якщо $n \notin N_B$;}\\
\inf{\Delta^{(\pm D, B)} _{c_1c_2...c_{n-1}c_n}}& = \sup{\Delta^{(\pm D, B)} _{c_1c_2...c_{n-1}[c_n+1]}},&\text{якщо $n \in N_B$},\\
\end{aligned}
\right.
$$
де $c_n \ne d_n-1$.
\item
$$
\bigcap^{\infty} _{n=1}{\Delta^{(\pm D, B)} _{c_1c_2...c_n}}=x=\Delta^{(\pm D, B)} _{c_1c_2...c_n...}.
$$
\end{enumerate}
\end{lemma}

\subsection{Знакозмінний $\tilde Q$-розклад і квазі-нега-$\tilde Q$-представлення.}

Нехай $\tilde Q=||q_{i,n}||$~--- фіксована матриця, де $n=1,2,...,$ $i=\overline{0,m_n}$, $m_n\in \mathbb N \cup \{0,\infty\}$, для якої справедливими є наступні  властивості: 
\begin{itemize}
\item$ \mathbb R \ni q_{i,n}>0$;
\item  для будь-якого $n \in \mathbb N$: $\sum^{m_n}_{i=0} {q_{i,n}}=~1$; 
\item для довільної послідовності $(i_n)$, $ i_n \in \mathbb N \cup \{0\}$: $\prod^{\infty} _{n=1} {q_{i_n,n}}=0$.
\end{itemize}

Квазі-нега-D-представлення дійсного числа можна узагальнити до розкладу в ряд
\begin{equation}
\label{eq: znakozminne tilde Q - predstavlennya 1}
(-1)^{\rho_1}a_{i_1,1}+\sum^{\infty} _{n=2}{\left[(-1)^{\rho_n}a_{i_n,n}\prod^{n-1} _{j=1} {q_{i_j,j}}\right]},
\end{equation}
де
$$
\rho_n=\begin{cases}
2,&\text{якщо $n \notin N_B$;}\\
1,&\text{якщо $n \in N_B$.}
\end{cases}
$$
або,  внісши кілька поправок, квазі-нега-D-представлення дійсного числа можна узагальнити до  розкладу дійсного числа в ряд
\begin{equation}
\label{eq: quasy-nega-tilde Q-representation 1}
2-\rho_1+(-1)^{\rho_1}\delta{'} _{i_1,1}+\sum^{\infty} _{n=2}{\left[(-1)^{\rho_n}\delta^{'} _{i_n,n}\prod^{n-1} _{j=1}{q^{'} _{i_j,j}}\right]}+\sum_{1<n\in N_B}{\left(\prod^{n-1} _{j=1}{q^{'} _{i_j,j}}\right)}, ~\mbox{де}  
\end{equation}
$$
 \delta^{'} _{i_n,n}=\begin{cases}
0,&\text{якщо $i_n=0$ та $n\notin N_B$;}\\
\sum^{i_n-1} _{i=0} {q_{i,n}},&\text{якщо $i_{n}\ne 0$  та $n\notin N_B$;}\\
\sum^{m_{n}} _{i=m_{n}-i_{n}} {q_{i,n}},&\text{якщо $n\in N_B$,}
\end{cases}
$$
 $$
q^{'} _{i_n,n}=\begin{cases}
q_{i_n,n},&\text{якщо $n \notin N_B$;}\\
q_{m_n-i_n,n},&\text{якщо $n \in N_B$,}
\end{cases}
$$

Розклад числа $x\in[0;1]$ в ряд \eqref{eq: quasy-nega-tilde Q-representation 1} називається \emph{квазі-нега-$\tilde Q$-представленням числа $x$}, а позначення $x=\Delta^{(\pm \tilde Q, B)} _{i_1i_2...i_n...}$ ---  \emph{квазі-нега-$\tilde Q$-зображенням числа $x$}.

\begin{lemma} Справедливою є слідуюча тотожність:
$$
\Delta^{(\pm \tilde Q, B)} _{i_1i_2...i_n...}\equiv \Delta^{\tilde Q} _{i^{'} _1i^{'} _2...i^{'} _n...},
$$
де
$$
i^{'} _n=\begin{cases}
i_n,&\text{якщо $n \notin N_B$;}\\
m_n-i_n,&\text{якщо $n \in N_B$,}
\end{cases}
$$
\end{lemma}
$\Delta^{\tilde Q} _{j_1j_2...j_n...}$ --- $\tilde Q$-зображення дійсного числа з $[0;1]$.

Тобто,
$$
x=a_{i^{'} _1,1}+\sum^{\infty} _{n=2}{\left[a_{i^{'} _n,n}\prod^{n-1} _{j=1}{q_{i^{'} _j,j}}\right]}.
$$

Тополого-метричні теорії представлення у чисел у вигляді розкладів в ряди \eqref{eq: znakozminne tilde Q - predstavlennya 1},~\eqref{eq: quasy-nega-tilde Q-representation 1} будуть описані в послідуючій статті.

Цілком очевидно, що коли $N_B=\varnothing$, отримуємо знакододатне представлення дійсних чисел (s-ве, канторівське чи $\tilde Q$-розклад). Коли ж $N_B$ є множиною лише парних або лише непарних натуральних  чисел, то отримуємо відповідне знакопочережне представлення.

Проведені вище міркування підштовхують до постановки та дослідження наступних задач:

\textbf{Задача 1.} \emph{Чи може існувати знакозмінний аналог деякого представлення дійсних чисел, якщо існуюють відповідні знакододатне та знакопочережні представлення?} Вище показано, що останнє є справедливим для s-го, канторівського та $\tilde Q$-представлень. Але чи існують інші такі представлення? Що можна сказати, наприклад, про представлення чисел рядом, елементи якого є числами, обернені до натуральних чисел? Наприклад, ряди Люрота.

Тобто, нехай, наприклад, маємо ряд виду
$$
\frac{(-1)^{\delta_1}}{a_1}+\sum_{n\ge 2}{\frac{(-1)^{\delta_n}}{a_1(a_1+1)...a_{n-1}(a_{n-1}+1)a_n}},
$$
де $a_n \in \mathbb N$, $N_B$ --- фіксована підмножина натуральних чисел, $B=(b_n)$ --- зростаюча послідовність всіх елементів з $N_B$,
$$
\delta_n=\begin{cases}
2,&\text{якщо $n \notin N_B$;}\\
1,&\text{якщо $n \in N_B$.}
\end{cases}
$$

\emph{Чи можна представити довільне число з деякого інтервалу у вигляді розкладу в  останній ряд?}  (Задача ускладнюється тим, що в знакододатному та знакопочережному рядах Люрота знаменники відповідних доданків не є однаковими). Якщо ні, то як треба видозмінити знаменник, щоб довільне число з деякого інтервалу  можна було представити у вигляді розкладу в отриманий знакозмінний ряд?

\textbf{Задача 2.} Як виявилось, ввівши кілька поправок, квазі-нега-s-кове, квазі-нега-D-,  квазі-нега-$\tilde Q$-представлення можна розглядати як системи числення, при означенні яких використано певний перетворювач цифр (символів). Тобто, нехай $N_B$ --- фіксована підмножина натуральних чисел, тоді:
\begin{itemize}
\item для s-го представлення
$$
[0;1]\ni x=\Delta^{\pm s, B} _{\alpha_1\alpha_2...\alpha_n...}\equiv  \Delta^{s} _{\varphi(\alpha_1)\varphi(\alpha_2)...\varphi(\alpha_n)...},
$$
$$
\varphi(\alpha_n)=\begin{cases}
\alpha_n,&\text{якщо $n \notin N_B$;}\\
s-1-\alpha_n,&\text{якщо $n \in N_B$.}
\end{cases}
$$
\item для канторівського представлення
$$
[0;1]\ni x=\Delta^{(\pm D, B)} _{\varepsilon_1\varepsilon_2...\varepsilon_n...}\equiv \Delta^D _{\varphi(\varepsilon_1)\varphi(\varepsilon_2)...\varphi(\varepsilon_n)...},
$$
$$
\varphi(\varepsilon_n)=\begin{cases}
\varepsilon_n,&\text{якщо $n \notin N_B$;}\\
d_n-1-\varepsilon_n,&\text{якщо $n \in N_B$.}
\end{cases}
$$
\item для $\tilde Q$-представлення
$$
[0;1]\ni x=\Delta^{(\pm \tilde Q, B)} _{i_1i_2...i_n...}\equiv\Delta^{\tilde Q} _{\varphi(i_1)\varphi(i_2)...\varphi(i_n)...},
$$
$$
\varphi(i_n)=\begin{cases}
i_n,&\text{якщо $n \notin N_B$;}\\
m_n-i_n,&\text{якщо $n \in N_B$.}
\end{cases}
$$
\end{itemize}

Як наслідок, виникає \emph{задача про застосування перетворюачів цифр (чи комбінацій цифр) до побудов систем числення та  побудову системи числення за наперед заданими геометричними властивостями.}

\begin{remark}
Довільну функцію $f$ спеціального виду, яка володіє певним набором властивостей, можна використовувати для моделювання  представлення дійсних чисел з деякого інтервалу. Зокрема, для моделювання представлень дійсних чисел з нульовою надлишковістю відзначимо такі набори властивостей функції $f$:
\begin{itemize}
\item функція $f$ --- неперервна і строго монотонна на деякому інтервалі;
\item функція $f$: неперервна майже скрізь на деякому інтервалі, що є її областю визначення; множиною значень функції $f$ є інтервал; майже скрізь в розумінні міри Лебега $\sharp\{x: f(x)=y_0 \}=1$, де $y_0$ --- де довільне фіксоване число з множини значень $f$.
\end{itemize}
\end{remark}

Розглянемо другий випадок. Використаємо функції, вивчені автором в  \cite{{S. Serbenyuk, functions with complicated local structure 2013}, {S. Serbenyuk abstract 6}, {S. Serbenyuk abstract 7}, {Symon12(2)}}.

 \section{Псевдо-s-ве зображення. Моделювання систем числення за наперед заданими геометричними властивостями. Найпростіші приклади}

\subsection{Окремі приклади моделювання представлень дійсних чисел  за наперед заданими геометричними властивостями таких представлень.}

\textbf{Приклад 1.}~Нехай маємо відрізок $[0;1]$. Нехай потрібно побудувати представлення дробової частини дійсного числа, при побудові якого на кожному кроці $n$, $n=1,2,3,...$, здійснюється поділ елементарного відрізка на $3^n$ однакових відрізків довжиною $\frac{1}{3^n}$. Причому, на кожному кроці $n$ розбиття  відрізка рангу $n-1$ циліндричний відрізок, що відповідає цифрі $0$ в моделюємому зображенні числа є першим зліва, а решта два (відповідні цифрам $1$ та $2$) розташовані ''справа наліво''.

Як наслідок, отримаємо наступне зображення
$$
\Delta^{3^{'}} _{\alpha_1\alpha_2...\alpha_n...}\equiv\Delta^3 _{\varphi(\alpha_1)\varphi(\alpha_2)...\varphi(\alpha_n)...}\equiv \sum^{\infty} _{n=1}{\frac{\varphi(\alpha_n)}{3^n}},
$$
де
\begin{center}
\begin{tabular}{|l|l|l|l|}
\hline
$\alpha$ & 0 & 1 & 2\\
\hline
$\varphi(\alpha)$ & 0 & 2 & 1\\
\hline
\end{tabular}
\end{center}
або, що еквівалентно
$$
\varphi(\alpha_n)=\begin{cases}
\alpha_n,&\text{якщо $\alpha_n=0$;}\\
s-1-\alpha_n,&\text{якщо $\alpha_n \ne 0$.}
\end{cases}
$$

Слід відмітити, що 
$$
\Delta^{3^{'}} _{\alpha_1\alpha_2...\alpha_{n-1}2(0)}=\Delta^{3^{'}} _{\alpha_1\alpha_2...\alpha_{n-1}0(1)},
$$
$$
\Delta^{3^{'}} _{\alpha_1\alpha_2...\alpha_{n-1}1(0)}=\Delta^{3^{'}} _{\alpha_1\alpha_2...\alpha_{n-1}2(1)}.
$$
Крім того,
$$
x=\sum^{\infty} _{k=1}{\frac{2-\gamma_k}{3^{a_1+a_2+...+a_k}}}\equiv \Delta^3 _{\underbrace{0...0}_{a_1-1}\varphi(\gamma_1)\underbrace{0...0}_{a_1+a_2-2}\varphi(\gamma_2)...\underbrace{0...0}_{a_1+a_2+...+a_k-k}\varphi(\gamma_k)...}.
$$
$\gamma_k \in \{1,2\}$, $(a_k)$ --- деяка послідовність натуральних чисел.

\textbf{Приклад 2.} Приклад періодичного використання перетворювачів цифр.
$$
\Delta^{3^{'}} _{\alpha_1\alpha_2...\alpha_n...}\equiv\Delta^3 _{O(\alpha_1)\varphi_1(\alpha_2)T(\alpha_3)O(\alpha_4)\varphi_1(\alpha_5)T(\alpha_6)...},
$$
де $O(\alpha_n)=\alpha_n$, $T(\alpha_n)=s-1-\alpha_n$, $\varphi_1$ --- деякий перетворювач цифр, причому $\varphi(i)\ne\varphi(j)$, $i\ne j$ $i,j\in\{0,1,2\}.$

У трійковій системі числення можна означити $m=3!=6$ різних перетворювачів цифр. Їх наведено  в наступній таблиці.
\begin{center}
\begin{tabular}{|c|c|c|c|}
\hline
 &$ $ 0 &$ 1 $ & $2$\\
\hline
$\varphi_1 (\alpha_n) $ &$0$ & $1$ & $2$\\
\hline
$\varphi_2 (\alpha_n) $ &$0$ & $2$ & $1$\\
\hline
$\varphi_3 (\alpha_n) $ &$1$ & $0$ & $2$\\
\hline
$\varphi_4 (\alpha_n) $ &$1$ & $2$ & $0$\\
\hline
$\varphi_5 (\alpha_n) $ &$2$ & $0$ & $1$\\
\hline
$\varphi_6 (\alpha_n) $ &$2$ & $1$ & $0$\\
\hline
\end{tabular}
\end{center}

\subsection{Псевдо-s-ве зображення.}

Нехай $1<s$  --- фіксоване число. 

\begin{definition}
\emph{Перетворювачем s-их цифр} називатимемо  відображення $\varphi$, областю визначення і множиною значень якого є на алфавіт $A\equiv \{0,1,...,s-1\}$ s-их цифр, таке, що $\varphi(i)\ne\varphi(j)$ для довільних $i, j$, для яких  $i\ne j$. 

 \emph{Перетворювачем $k$-цифрових наборів  s-их цифр} називатимемо  функцію $\varphi$  $k$ змінних, областю визначення і множиною значень якої є декартів добуток
$$
A^k=\underbrace{A\times A\times ... \times A}_k
$$
алфавітів s-их цифр, така, що $\varphi(i)\ne\varphi(j)$ для  довільних $i, j$, для яких  $i\ne j$.
\end{definition}

\begin{definition}
Зображення дійсних чисел $\Delta^{\hat{s}} _{\alpha_1\alpha_2...\alpha_n...}$, яке взаємопов'язане з s-им зображенням наступним чином 
$$
\Delta^{\hat{s}} _{\alpha_1\alpha_2...\alpha_n...}\equiv\Delta^s _{\varphi(\alpha_1)\varphi(\alpha_2)...\varphi(\alpha_n)...}\equiv \sum^{\infty} _{n=1}{\frac{\varphi(\alpha_n)}{s^n}},
$$
де $\varphi$ --- деякий перетворювач s-их цифр (одноцифрових наборів s-их цифр), називатимемо \emph{псевдо-s-им зображенням з набором параметрів $\{(1,1)\}$}. 

Зображення дійсних чисел $\Delta^{\hat{s}} _{\alpha_1\alpha_2...\alpha_n...}$, яке взаємопов'язане з s-им зображенням наступним чином 
$$
\Delta^{\hat{s}} _{\alpha_1\alpha_2...\alpha_n...}\equiv\Delta^s _{\varphi(\alpha_1...\alpha_k)\varphi(\alpha_{k+1}...\alpha_{2k})...\varphi(\alpha_{kn+1}...\alpha_{(n+1)k})...},
$$
де $\varphi$ --- деякий перетворювач $k$-цифрових наборів s-их цифр, називатимемо \emph{псевдо-s-им зображенням з набором параметрів $\{(1,k)\}$}. 

Зображення дійсних чисел $\Delta^{\hat{s}} _{\alpha_1\alpha_2...\alpha_n...}$, яке взаємопов'язане з s-им зображенням наступним чином 
$$
\Delta^{\hat{s}} _{\alpha_1\alpha_2...\alpha_n...}\equiv\Delta^s _{(\varphi_1(\alpha_1...\alpha_{k_1})\varphi_2(\alpha_{k_1+1}...\alpha_{k_2})...\varphi_m(\alpha_{k_{m-1}+1}...\alpha_{k_m}))},
$$
де під круглими дужками розуміється період,  $\varphi_t$ --- деякий перетворювач $k_t$-цифрових наборів s-их цифр, $t=\overline{1,m}$, $m$ --- деяке фіксоване натуральне число, $k_1,...,k_m$ --- фіксований набір натуральних чисел, називатимемо \emph{псевдо-s-им зображенням з набором параметрів $\{(m; k_1, k_2,..., k_m)\}$}. 
\end{definition}

\begin{theorem}
Метричні та фрактальні властивості псевдо-s-го та  s-го представлень співпадають.
\end{theorem}

В послідуючих статтях подібним чином буде узагальнено до псевдо-D-зображення та псевдо-$\tilde Q$-зображення D- та $\tilde Q$-зображення дійсних чисел.

\end{document}